\documentclass[12pt]{iopart}

\usepackage{iopams}
\usepackage{enumerate}
\usepackage{url}

\usepackage[colorinlistoftodos]{todonotes}

\expandafter\let\csname equation*\endcsname\relax
\expandafter\let\csname endequation*\endcsname\relax
\usepackage{amsmath}
\usepackage{amssymb}
\usepackage{tikz-cd}
\usepackage{graphicx}

\usepackage[margin=0cm]{caption}

\usepackage{amsthm}

\usepackage{placeins}

\usepackage{verbatim}

\newtheorem{theorem}{Theorem}[section]

\newtheorem{definition}{Definition}[section]

\begin{document}

\title[Control landscapes for a class of non-linear dynamical systems]{Control landscapes for a class of non-linear dynamical systems: sufficient conditions for the absence of traps}

\author{Benjamin Russell}
\address{Department Of Chemistry, Princeton University, Princeton, NJ 08544, USA}
\ead{br6@princeton.edu}

\author{Shanon Vuglar}
\address{Department Of Chemistry, Princeton University, Princeton, NJ 08544, USA}
\ead{svuglar@princeton.edu}

\author{Herschel Rabitz}
\address{Department Of Chemistry, Princeton University, Princeton, NJ 08544, USA}
\ead{hrabitz@princeton.edu}

\begin{abstract}
We establish three tractable, jointly sufficient conditions for the control landscapes of non-linear control systems to be trap free comparable to those now well known in quantum control.
In particular, our results encompass end-point control problems for a general class of non-linear control systems of the form of a linear time invariant term with an additional state dependent non-linear term.
Trap free landscapes ensure that local optimization methods (such as gradient ascent) can achieve monotonic convergence to effective control schemes in both simulation and practice.
Within a large class of non-linear control problems, each of the three conditions is shown to hold for all but a null set of cases.
Furthermore, we establish a Lipschitz condition for two of these assumptions; under specific circumstances, we explicitly find the associated Lipschitz constants.
A detailed numerical investigation using the D-MOPRH control optimization algorithm is presented for a specific family of systems which meet the conditions for possessing trap free control landscapes.
The results obtained confirm the trap free nature of the landscapes of such systems.






\end{abstract} 

\maketitle

\section{Introduction}

\subsection{Introduction to Control Landscapes}

Control landscapes were introduced via the study of quantum control \cite{chakrabarti2007quantum,rabitz2000whither,rabitz2004quantum,moore2012laser,rabitz2003shaped} following the observation that local optimization techniques, for example gradient ascent, typically succeed in shaping effective  laser pulses for a variety of control objectives.
Generalizing from the quantum setting, a control landscape is an objective function which depends on the response of the dynamics of the system under study as a function of the control.
The study of control landscapes involves analyzing the structure of the critical points of such composite functions.

Non-linear control finds broad application in engineering and mathematics \cite{slotine1991applied, isidori2013nonlinear} and non-linear differential equations have, in general, very broad application in physics, biology, chemistry, and in engineering \cite{strogatz2014nonlinear, kaplan2012understanding, wiggins2003introduction, nayfeh2008applied}.
Obtaining maximally general mathematically sufficient criteria for a given system to possess a trap free control landscape is therefore of some importance.
Accordingly, this work extends recent results on quantum control landscapes to a class of non-linear control systems on arbitrary manifolds.

\subsection{Summary of Main Results}

We consider controlling systems of the form
\begin{eqnarray}
  \frac{d x(t)}{dt} = F(x(t),w(t))
\end{eqnarray}
which evolve on a smooth (which will henceforth be assumed without explicit statement) manifold $M$ under the influence of a smooth control $w$.
The fidelity of the control is evaluated using a smooth cost function $J: M \rightarrow [0,1]$.
We show that three specific assumptions about the function $F$ and final time $T$, and a single assumption about the cost function $J$, are sufficient for the overall fidelity $F(w) := J(V_T[w, x_0])$ to be free from local optima.
Here, $V_T$ (known as the end-point map) maps a control to the state that solves the underlying dynamical equation above and $x_0$ is the system's initial state.
These assumptions guarantee that effective controls $w$ can be found using local optimization algorithms.
We go on to assess under which conditions, and how typically, each of these assumptions holds.

The related concept of a fitness landscape (as introduced in evolutionary biology in \cite{wright1932roles}) has, alongside it's now widespread use in biology, led to the development of a variety of novel approaches to optimization such as genetic algorithms and related methods \cite{brownlee2011clever}.
The contradistinction between the concept of a control landscape as studied in this work and the fitness landscapes introduced in \cite{wright1932roles} is a largely conceptual but important distinction.
In the biological case, the landscape is the biological fitness as a function of the genome of an organism and is concerned with the autonomous changes in the dependent variables rather than changes to control variables to which experimenters have access.

\section{Control Landscapes}

The desire to control quantum systems using shaped electromagnetic radiation to implement specific unitary transformations \cite{hsieh2008optimal}, create specific quantum states \cite{weinacht1999controlling,turinici2001quantum}, or control a specific quantum observable \cite{wu2008control,chakrabarti2008quantum,nanduri2013exploring} is driven by several goals including, but by no means limited to, the selective breaking of chemical bonds \cite{tibbetts2013optimal,shapiro2003principles,seideman1989coherent}, the control of chemical reactions more generally \cite{shapiro1997quantum}, time optimal quantum computation \cite{carlini2007time,carlini2006time,wang2015quantum,russell2015zermelo,russell2014zermelo,brody2015time,brody2015solution,brody2015elementary,khaneja2001time,arenzspeed}, and various others \cite{dong2010quantum}.

Initially there was a widespread expectation that effective quantum control schemes would be difficult to either discover or design.
This perspective stemmed from both the counterintuitive nature of the response of quantum dynamics to stimulation by external control fields and from the fact that the mathematical space of all possible control fields is of very high dimension and was common throughout quantum control in theory, numerical simulation, and experimental practice.
The former observation was believed to suggest that the control landscapes of quantum systems, defined as the objective fidelity as a function of the available control variables, would be a complicated function typically possessing many local optima.
Such local optima would preclude the possibility of using learning control with local gradient based algorithms to find high fidelity pulse shapes as they are susceptible to converging to local optima.
The latter point was believed to indicate that the \emph{curse of dimensionality} would mar search efforts and render them intractable. 

As a large body of numerical results \cite{riviello2014searching, sun2015measuring} (and later experimental results \cite{sun2014experimental,sun2015experimental}) were gathered, it became clear that gradient ascent and other local optimization methods \cite{khaneja2005optimal,maday2003new,schirmer2011efficient,krotov1983iterative,machnes2016gradient} for obtaining effective shaped laser control pulses were successful in the vast majority of applications.
In light of these results, explanations were put forth for why effective quantum controls were easier to find than initially anticipated \cite{ho2006effective}.
Furthermore, questions were posed as to the fullest possible scope of this body of results, including applications in controlled chemical, physical, and biological processes as well as many engineering applications \cite{rabitz2012control}.

Closed finite dimensional systems with linear coupling to a control field $E(t)$ have been found to behave favorably using gradient methods.
Such systems are governed by the following form of Schr\"odinger equation
\begin{align}
	\frac{d U_t}{dt} = \left(iH_0 + E(t)i H_c \right)U_t
\end{align}
where $H_0$ represents the free systems dynamics in the absence of control and $H_c$ represents the coupling to the control field.
This form is known in the physical sciences as the dipole approximation and in mathematics as a right invariant, affine bilinear control system on the unitary group.
Such equations are of particular interest in quantum control for quantum computation as this class includes models of systems such as those found in NMR for which first order coupling to a control field is dominant.

The unanticipated success of simple control optimization methods in fulfilling varied quantum control desiderata led to a detailed theoretical investigation into the topology of the set of critical points in control space, i.e., controls for which the gradient of the fidelity is zero \cite{wu2012singularities}.
It was found that three conditions on a quantum system are sufficient for the associated quantum control landscape to have the same critical point structure as the chosen cost function $J$.
These conditions are:
\begin{enumerate}
  \item Controllability -- every goal state can be implemented by some control $w$.
  \item Local controllability -- the end-point state can be freely, infinitesimally varied by varying the control $w \mapsto w + \delta w$.
  \item Unconstrained control resources -- $w$ can be any smooth function.
\end{enumerate}
To explain the pervasive success of simple optimization techniques observed in quantum control, the range of validity for each of the three assumptions was assessed.

The first assumption, controllability, was shown to hold for \emph{almost all} (in the measure theoretic sense \cite{halmos2014measure}) quantum systems in the dipole approximation \cite{altafini2002controllability} (with one or more control fields).
It was later shown to hold for almost all systems only possessing two body internal interactions \cite{takui2016electron, 2017arXiv170100216M}.
Furthermore and analogously, a significantly relaxed version of the second assumption has been shown to hold for almost all controlled quantum systems \cite{aatfme}.
However, a result analogous to the case of systems possessing two body interactions alone has not yet been forthcoming for the the second assumption.

A central conclusion of the culminated investigation into quantum control landscapes hitherto is that limitations on control resources are typically the determining factor as to whether landscape traps are present.
Hence, the efficacy of gradient based (or other local) control optimization algorithms is largely contingent upon the available control resources.

\subsection{Controlled Systems}

Geometric control theory \cite{jurdjevic2008geometric} is concerned with a far broader class of dynamical systems than closed, finite dimensional quantum systems with linear coupling to a control field.
In this work, we extent both the notions and some results from quantum control systems to a class of non-linear control systems.
This goal is motivated by the observation that many (but not all) of the known properties of the control landscapes of quantum systems do not depend critically on the unitary character of quantum dynamics.
Rather, they depend more strongly on the underlying smooth structure of the manifold of the unitary group $U(n)$ (as opposed to the algebraic structure of this manifold as a Lie group).
We focus on control of non-linear systems of the following form:
\begin{definition}{}
  A \emph{first order control system} (FOCS) on a manifold $M$ is given by:
  \begin{eqnarray}
      \label{GFOCS}
      \frac{d x(t)}{dt} = F(x(t),w(t))
  \end{eqnarray}
wherein $F \in \Gamma[M]$ is a smooth vector field.
Henceforth and standardly \cite{leeintro} $\Gamma(M)$ denotes the set of all smooth vector fields on $M$.

\end{definition}

\subsection{Cost Functions}
\label{costfunction}

Cost functions can be broadly categorized as possessing two types of term, namely, run-time and terminal costs.
Runtime costs contain `fluence' type terms such as $\int_0^T w(t)^2 dt$, whereas, terminal costs depend only on the end-point $x(T) = V_T[w, x_0]$ (and not directly on $w(t)$ during an evolution).
Unlike in typical quantum control applications, in engineering applications it is often vital to minimize a certain run-time cost \cite{geering2007optimal}.
A canonical example is that of LTI systems with a quadratic run-time cost, known as the linear quadratic regulator problem \cite{kwakernaak1972linear}.
However, we restrict attention to the case of finding a control which drives the state to a desired goal at a prespecified terminal time; only cost functions corresponding to terminal costs are considered.
We present this as a first step towards broadening the scope of control landscape analysis and expect that future work will incorporate additional cost function types.

We consider cost functions $J:M \rightarrow [0,1]$ which have the following properties:
\begin{enumerate}
  \item smooth;
  \item possessing only global maxima, global minima, and possibly saddle type critical points (i.e., they posses no local optima);
  \item possessing a unique global maximum.
\end{enumerate}
Cost functions which have all three properties will be referred to as \emph{admissible} throughout this work.
Unless otherwise indicated, the term \emph{cost function} will henceforth be used to refer  to admissible cost functions which depend only on the system's end-point.

\subsection{Control Landscapes}

A control landscape for a system of the form $\eqref{GFOCS}$ is the value of the cost function as a function of the control: $J(V_T[w, x_0])$.
\begin{figure}[!h]
  \begin{minipage}[t]{0.45\textwidth}
    \includegraphics[width=\textwidth]{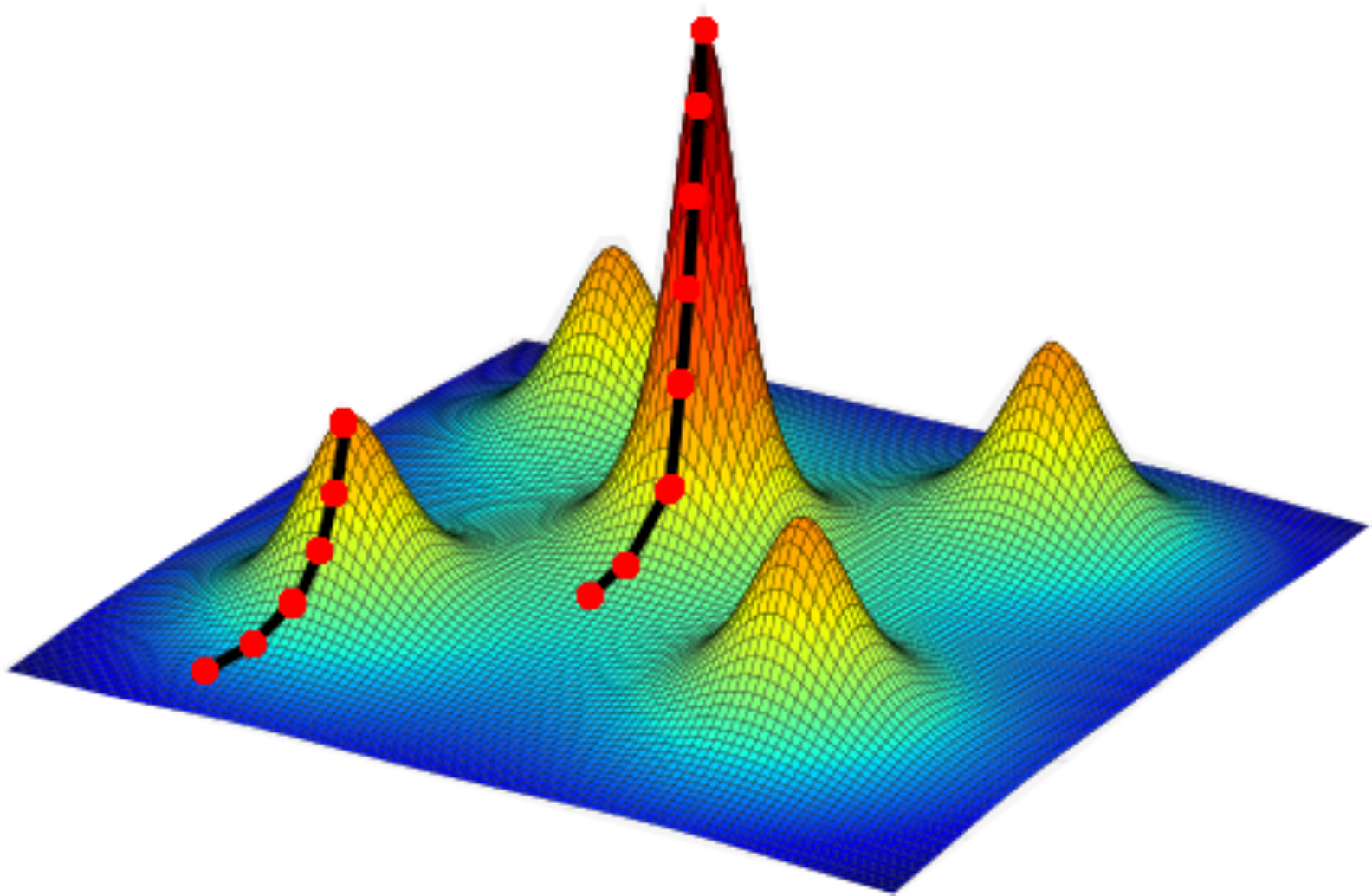}
	\caption{A cost function with two variables and a single global maximum but which also has local optima. These are shown trapping a gradient ascent (algorithm iterations shown by the red dots) optimization when two different initial values are chosen.}
\label{trapsandnot1}
\end{minipage}
  \hfill
  \begin{minipage}[t]{0.45\textwidth}
    \includegraphics[width=\textwidth]{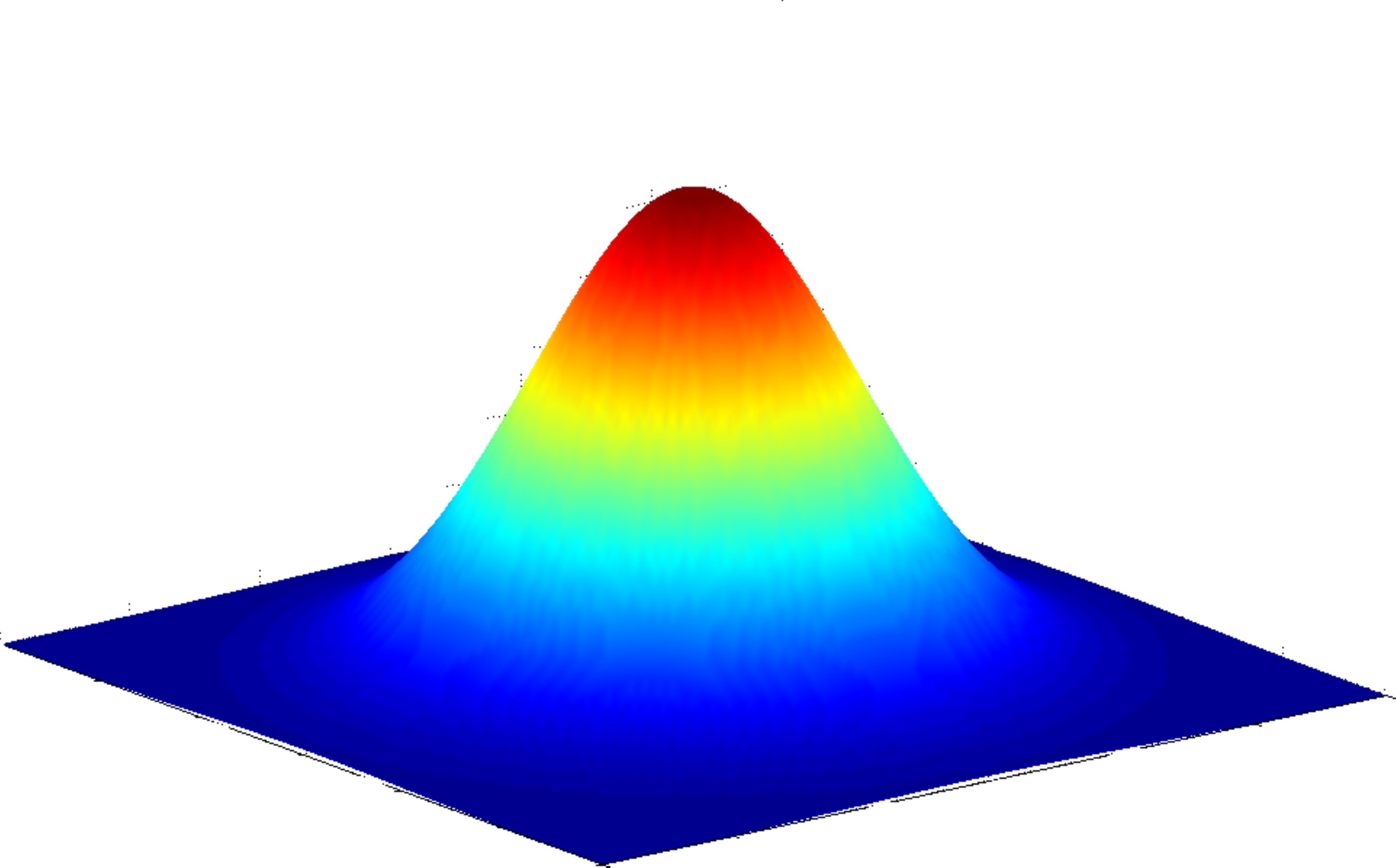}
	\caption{A cost function with two variables, a single global maximum, and with no local optima.}
	\label{trapsandnot2}    
  \end{minipage}
\end{figure}
Figures \ref{trapsandnot1} and \ref{trapsandnot2} show two juxtaposed control landscapes, one with traps and one without.
However, it is noteworthy that these figures show landscapes which depend on only two control variables, not on a time dependent control function as that scenario would be impossible to illustrate informatively.

\subsection{Traps and Critical Points}
\label{sec:crit}

The question of whether or not a given control landscape possesses traps is of crucial importance in determining the potential for successfully discovering effective controls using local optimization techniques.
Traps are defined as local optima -- at such points in control space many local optimization algorithms will terminate having converged to a sub-optimal control.
For example, gradient ascent will terminate at such points as they have fidelity gradient zero.
However, a landscape free from traps is insufficient to reach the full conclusion that convergence to a globally optimal control will be achieved.
Landscapes free from local optima may possess saddle type critical points which can encumber a local optimization algorithm.

A control landscape for a system of the form \eqref{GFOCS} has two possible types of critical points.
This can be observed directly from the functional chain rule:
\begin{align}
	\frac{\delta J}{\delta w} = \frac{\delta J}{\delta V_T} \circ \frac{\delta V_T}{\delta w} = 0.
\end{align}
In order for a control to be a critical point on the landscape it must satisfy: $\frac{\delta J}{\delta w} = 0$.
This can happen in two ways, neither of which excludes the other, which we now introduce.
\begin{definition}[]
  Controls satisfying $\frac{\delta J}{\delta V_T} = 0$ are defined as \emph{regular critical points}.
\end{definition}
These controls drive the system's end-point to a point for which the cost function $J$ is critical.
Note that if a critical point of $J$ were a saddle, then any control driving the system's end-point to that point
would itself be a saddle on the control landscape.
\begin{definition}[]
  Controls satisfying $\frac{\delta V_T}{\delta w} = 0$ are defined as \emph{singular critical points}.
\end{definition}
These controls are those for which it is not possible to infinitesimally vary the end-point map in the direction required to increase the cost function by infinitesimally varying the control.
More exactly, at such controls $d V_T\big|_{w} $ fails to be full rank and $\nabla J \big|_{V_T[w, x_0]}$ fails to be in the image of this derivative.

In quantum control, this latter type of critical point has been the subject of significant debate \cite{pechen2011there,rabitz2012comment,de2013closer} wherein it was conjectured by some that they could introduce true traps into the landscape \cite{pechen2011there,rabitz2012comment}.
No examples of such singular controls introducing traps have currently been proven to exist.
However, many cases for which such singular controls, i.e., singular critical controls, do introduce saddles into the control landscape have been constructed or numerically discovered \cite{wu2008control}.
These properties have primarily, but not exclusively, been shown for controls which are constant, or even zero, as functions of time;
such constant controls are by no means the only singular critical controls \cite{de2013closer,PhysRevA.86.013405}.
The neighborhoods of singular critical controls have been explored in numerical studies of quantum control landscapes \cite{riviello2014searching}.
It was found that that the radius of the basin of attraction (under gradient flow in control space) for each of the numerically obtained singular controls has, at most, an extremely small volume in control space.

\subsection{Example Cost Functions}

Depending on the manifold underlying the system in question, different choices of cost function can be appropriate.
However, in light of the discussion in previous sections, some cost functions are more amenable to creating trap free control landscapes than others.
Specifically, it is desirable to have a cost function $J: M \rightarrow [0,1]$ which is free from traps, as traps in the cost function result in traps at any corresponding points in control space.
If $p \in M$ is a local optimum of $J$, then any control $w$ s.t. $V_T[w, x_0] = p$ will be a regular critical point on the control landscape.

In light of these distinctions and to illustrate concepts, we describe two important cost functions and some of their properties.
In the case of quantum control with a goal propagator $G \in U(n)$, the cost function on the unitary group $J: U(n) \rightarrow [0, 1]$ is typically given by either $J(U) = \text{Re}\{Tr(U G^\dagger)\}$ or $J(U) = \big|Tr(U G^\dagger)\big|^2$.
Both of these functions have saddle points at any $U = G V^\dagger \begin{pmatrix} \pm 1 & \cdots & 0 \\ \vdots & \ddots & \vdots \\ 0 & \ldots & \pm 1 \end{pmatrix} V$ for any unitary $V$.
For $J(U) = \big|Tr(U^\dagger G)\big|^2$ there is only a single global optimum at $U=G$ upto a global change  in phase $U \mapsto e^{i \theta}U$.
For a clear discussion of the nature of the critical points of these two important cost functions in quantum control see \cite{de2013closer}.
The second cost function described above has a clear analogue on any compact, connected Lie group as it is (upto a constant multiple) given simply by the unique distance function associated to the bi-invariant metric.
While it is not the only possible choice on such spaces, it is a clear logical candidate with many favorable features.
Specifically, being bi-invariant, it treats all parts of the group `the same' in a sense strongly, but not wholly, analogous to the translation invariance of the Euclidean norm in $\mathbb{R}^N$.
In the case of a control system on $\mathbb{R}^{n}$, the Euclidean metric is a good choice for similar reasons, including translation invariance.
This norm yields the cost function $J(x) = ||x - G||$ wherein $G\in \mathbb{R}^n$ is the goal.
This cost function is highly favorable; it has no saddle points, no local optima, and a unique global optimum at the goal $G$.

\subsection{Augmented End-Point Map}

In following sections, it will become useful to understand the \emph{generic} properties of families of non-linear control landscapes.
We introduce the concept of \emph{structural parameters} (denoted $\lambda$) as distinct from control parameters defining a control $w$.
In the interest of generality, we take $\lambda$ to be drawn from a smooth manifold $S$ unless otherwise stated.
Accordingly, we consider parameterized families of non-linear control systems of the form:
\begin{eqnarray}
	\frac{d x(t)}{dt} = F(x(t), w(t), \lambda).
\end{eqnarray}
For such a family with parameters $\lambda \in S$ we call a property of the associated non-linear control landscape \emph{generic} if it holds for all $\lambda \in S$ other than for a null set of $\lambda$ values.

\section{Sufficient Conditions for a Trap Free Control Landscape in Non-Linear Control Systems}

Given a system of the form (\ref{GFOCS}), there are three assumptions which are jointly sufficient conditions for a trap free control landscape.
In this section, we first introduce the required definitions to state these assumptions and then go on to prove that they are sufficient conditions.

\subsection{Controllability}

\begin{definition}{}
	A system is called \emph{globally controllable} (or simply \emph{controllable}) if for every pair of points $p_0, p_1 \in M$ there exists a time $T$ and a control $w$ such that $V_T[w,p_0] = p_1$.
\end{definition}

\begin{definition}{}
A system is called \emph{globally fixed time controllable} (or simply \emph{fixed time controllable}) for a given time $T$ if for every pair of points $p_0, p_1 \in M$ there exists a control $w$ such that $V_T[w,p_0] = p_1$.
\end{definition}

Equivalently, fixed time controllability is the requirement that the end-point map $V_T$ be globally surjective on $M$ as a functional of $w$ where $w$ ranges over all smooth functions of time without restriction.

\subsection{Local Surjectivity/Transversality}

\begin{definition}
A system is called \emph{locally controllable} or \emph{locally surjective} if there are no singular controls.
The control $w$ is a singular control if $d V_T\Big|_{V_T[w,x_0]}$ is not full rank, i.e., its image is not full in the tangent space $T_{V_T[w,x_0]}M$, at the point $V_T[w,x_0] \in M$.
\end{definition}

\begin{definition}
A system is called \emph{transverse} (relative to a given submanifold $Q \subset M$) if there are no controls such that $\langle q, \ \delta V_T[\delta w] \rangle = 0$ for all $\delta w$ whenever $q$ is a tangent vector to $Q$.
This is equivalent to saying that there are no singular critical controls driving the system to a point within the submanifold $Q$.
\end{definition}
This second condition is far weaker than the the first.
Transversality only requires that it be possible to maneuver the end-point in at least one direction with a component orthogonal to the tangent space to $Q$.
In the case where $Q$ is taken to be a level set of the cost function $J$ this condition becomes equivalent to the existence of a variation $\delta w$ which maneuvers the end-point in a direction of increasing the cost function, i.e., not orthogonal to the gradient $\nabla J$ of the cost function.
This is opposed to local surjectivity which requires that \emph{all} directions are available, i.e., that the rank of $dV_T$ is full everywhere in the control space.

\subsection{Sufficient Resources}

In practice, $w$ is typically restricted to be drawn from some prespecified set in accordance with physical restrictions on the device implementing the control scheme or the device under control.
For example, $w$ may need to be smooth in the case of representing a physical actuator or bounded in the case of representing a laser field (in order to avoid damage to one's sample).
A typical set of restrictions is that of being bounded, smooth, and bandwidth limited.
However, it is also common to consider piecewise constant controls as these are typical in devices like frequency comb lasers (furthermore, they form good approximations to $L^p$ class functions \cite{levy2012elements} while also being significantly easier to work with).
A control systems is said to have sufficient resources if $w$ is allowed to vary within a domain such that that the end-point map $V_T$ is globally surjective on $M$ as a function of $w$.
This is to say, within a set of controls taken to be admissible the end-point map remains globally surjective on $M$.

\subsection{Sufficiency of The Three Assumptions of Control Landscape Analysis}

In this section, we present a proof that the three assumptions are sufficient for a trap free control landscape for all systems of the form \eqref{GFOCS}.
To be mathematically precise, we show that the gradient flow in control space converges to a globally optimal control, i.e., a regular critical point at the maximum value of the given cost function, regardless of the initial control chosen.

\begin{theorem}[The Three Assumptions are Sufficient for a Trap Free Landscape]
\label{sec:thm}
Consider a control system of the form \eqref{GFOCS} on a manifold $M$ and an admissible cost function $J: M\rightarrow [0,1]$ which together meet the three assumptions:
\begin{enumerate}
	\item The system is fixed time controllable for a given final time $T$.
	\item The system is locally controllable.
	\item The control field(s) are unrestricted, i.e., all smooth functions $w$ on $[0,T]$ are admissible controls.
\end{enumerate}
For such a system, gradient ascent will converge to the set of maxima of $J(V_T[w, x_0])$.
\end{theorem}

\begin{proof}
Consider the gradient flow of $J(V_T[w, x_0])$ in control space for a fixed but arbitrary $x_0 \in M$.
This is achieved via a smoothly parameterized family of controls in control space $w_s(t)$, parameterized by a real parameter $s$.
The flow is given by
\begin{align}
	\frac{\partial w_s(t)}{\partial s} = \nabla \left(J(V_T[w_s, x_0]) \right).
\end{align}
Assumptions 1 and 3 together are equivalent to saying that $V_T$, the end-point map, is globally surjective on $M$.
Formally, this is: $\forall x_1 \in M, \ \exists w \ \ \text{s.t.} \ \ V_T[w, x_0] = x_1$.
Assumption 2 is equivalent to saying that $V_T$, the end-point map, is locally surjective.
Formally, this is: $\forall w, d V_T\big|_{w}[\delta w]$ is full rank in $T_{V_T[w, x_0]}M$.

The critical points of $J(V_T[w, x_0])$ can now be assessed.
Recall that critical points can be classified as being of two (not mutually exclusive) types, singular and regular, as described in Sec. \ref{sec:crit}.
By the assumption of an admissible cost function, any point $p \in M$ for which $\nabla J\big|_{p}=0$ is either a global maximum, global minimum or a saddle.
Subsequently, by the assumption of local subjectivity, all controls driving the system's end-point to such a point $p\in M$ are also saddles as they can be infinitesimally varied to induce a change in the end-point which increases (or decreases) $J$.
Singular critical controls, i.e., controls $w$ at which $\langle \nabla J\big|_{V_T[w, x_0]}, \delta V_T \big|_{w}[\delta w] \rangle = 0, \ \forall \delta w$ are excluded by assumption 2.
Accordingly, the gradient flow will not encounter any controls for which $\nabla J(V_T[w, x_0])=0$ other than global maxima or minima or saddles on the control landscape.
\end{proof}

We note that, under the stated premises, this proof has excluded the existence of sub-optimal critical points other than saddles.
However, this doesn't directly assure that gradient ascent converges.
Additional caveats are required to ensure a gradient flow converges \cite{cortes2006finite}, however, it is assured that global maxima are stable critical points.
This is nothing other than the typical situation faced during any application of gradient flow in optimization.


\section{Assumption 1: Properties of Controllability}

\subsection{Controllability is Prolific}

Controllability of non-linear systems has been well studied from a mathematical perspective \cite{nijmeijer2013nonlinear, sussmann1972controllability, brockett1979feedback, hermann1977nonlinear, su1982linear, hermann1977nonlinear, klamka1996constrained, van1982observability, hirschorn1976global}.
Informally, for systems in $\mathbb{R}^n$, if the linear part of a system is controllable and the non-linear part is bounded (so as not to dominate the linear part), then the the overall system is controllable.

\subsection{Controllability of Linear Systems is Generic}

Conditions for the controllability of LTI systems 
\begin{align}
\label{LTI}
\frac{d X(t)}{dt} = AX(t) + Bw(t)
\end{align}
are well understood.
For example, the set of vectors
\begin{align}
	\{ B, AB, A^2B, \ldots, A^{(N-1)}B \}
\end{align}
forms a basis for $\mathbb{R}^N$ (and equivalently $B$ is a \emph{cyclic vector} for $A$)
if and only if the system \eqref{LTI} is controllable. 
It was elaborated in \cite{Russell20160210}, and elsewhere, that this holds for all $A,B$ other than a null set.
That is, if $A,B$ are chosen randomly, then with probability of $1$ the resulting system will be controllable.
Furthermore, for any non-controllable system there exists a small perturbation to $A$ and/or $B$ which is sufficient to restore controllability.

\subsection{Controllability of Non-Linear Perturbations to Linear systems}

The LTI case can be used as a toehold for understanding the controllability of a large class of non-linear control problems in $\mathbb{R}^N$ \cite{lukes1972global}.
The critical observation applies for systems of the form
\begin{align}
	\label{lukesform}
	\frac{d x(t)}{dt} = Ax(t) + Bw(t) + f(x(t), w(t)),
\end{align}
wherein $f$ is any smooth and bounded function.
It is shown in the central result of \cite{lukes1972global} that if the `linear part' of the system \eqref{lukesform} is controllable and the non-linear part, given by $f$, is smooth and bounded, then the system is controllable.
Furthermore, boundedness is shown to be sufficient but not necessary; a weaker Lipschitz condition on $f$ is sufficient to infer the controllability of \eqref{lukesform} from the controllability of the linear part.
Intuitively, this result is saying that if the linear part is controllable and the non-linear part fails to dominate the linear part, then the combination of the two parts retains controllability.
This result, in combination with the observation that almost all LTI systems are controllable, demonstrates that a very large class of non-linear control systems are controllable.

\section{Assumption 2: Two Approaches}
\subsection{Applying the Parametric Transversality Theorem to the Local Controllability Assumption}
\label{sec:ptt}
The parametric transversality theorem (or Thom's lemma) \cite{thom1956lemme} informally states that within a parameterized family of smooth maps between smooth manifolds, those maps transverse to a given submanifold of their target manifold, under a certain condition on the family of maps, are generic.
Which is to say, the maps failing this condition comprise all but a null set (the role of this theorem in quantum control is discussed in \cite{aatfme}).
A smooth map of manifolds $\phi: M \rightarrow N$ is called \emph{transverse} to a sub-manifold $Q < N$ (denoted $\phi \pitchfork Q$) if
\begin{align}
	\text{Im}\left(d \phi \big|_p\right) \oplus T_p Q \equiv T_p N.
\end{align}

We now state Thom's lemma more formally: if a family of maps is defined by $\phi_s(p) = \psi(s,p)$ where $\phi: M\times S \rightarrow N$ has the property that $\psi \pitchfork Q$, then $\phi_s \pitchfork Q$ for all but a null set of $s \in S$.
In terms of a control system, $V_T$ being transverse to a given submanifold of the state space, say a level set of fidelity, implies that there are no singular critical points on that sub-manifold.
This is due to the fact that the end-point can be steered away from that submanifold (and thus up the landscape in the case of level sets of fidelity) by an infinitesimal change in the control.
Thus, if a parameterized family of control systems have no traps in the preimage (in control space) of a given submanifold $Q$ of the systems state space, then only a null set of systems from the parameterized family can have a trap in the preimage of the same submanifold $Q$.

Consider a parameterized family of control systems
\begin{eqnarray}
	\frac{d x(t)}{dt} = F(x(t), w(t), s)
\end{eqnarray}
with $s \in S$ such that $d V_T$ is locally surjective (when both $s$ and $w$ are infinitesimally varied).
Such a system clearly has $V_T \pitchfork L_{\alpha}$ (depending on $s$ also) for each level set $L_{\alpha}$ of the cost function (in fact for every submanifold $Q < N$ on the system's state space).
This implies that each level set can only possess a trap for a null set of $s \in S$.
Furthermore, a countable union of null sets is itself null \cite{halmos2014measure}, therefore \emph{any} countable set of submanifolds can only contain a trap for a null set of $s \in S$.
Such a possibility is, in this sense, both rare and sensitive to infinitesimal changes in $s$.

Generic transversality notwithstanding, it is in principle possible for the null sets associated to each level set (or any covering set of submanifolds, such as any foliation) to union to a non-null set (or even the whole control space).
This possibility, which will form the basis of future work, seems highly implausible but cannot be completely excluded on the basis that $V_T$ is locally surjective (with both $s$ and $w$ permitted to vary).

\subsection{A Lipschitz Condition for Local Controllability}

While the geometric genericness results of Sec.\eqref{sec:ptt} allow the assessment of the typical landscape critical point structure within families of systems, these tools are not always needed; in many cases a stronger conclusion, which involves no probabilistic statements, can be made.
We will proceed, as in the case of assumption 1, to use LTI systems as a toehold in order to derive a Lipschitz condition sufficient for local controlability.
In the case of planar systems, we explicitly obtain one such Lipschitz constant.
This Lipschitz condition is on the norm of the \emph{Jacobian} of the non-linear part of the system, rather than on the non-linear part itself as in the case of assumption 1.

As described in \cite{Russell20160210} and elsewhere, local and global controllability are equivalent for LTI systems.
Accordingly, within the class of LTI systems, local controllability holds with probability one because global controllability holds with probability one.
We now specialize to a result for planar systems for which the non-linearity only involves the state, rather than both the state and the control.
Obtaining similar Lipschitz conditions for a wider range of systems will be part of future work.

Consider a planar system of the form
\begin{align}
	\label{eqn:ourform}
	\frac{d x}{d t} = Ax(t) + Bw(t) + f(x(t)).
\end{align}
Apply the variation $w \mapsto w + \delta w$, $x \mapsto x + \delta x$ in order to find the differential equation obeyed by $\delta x$
\begin{align}
	\label{eqn:ourformloc}
	\frac{d (\delta x)}{d t} = \left(A + (Df)(t) \right)\delta x(t) + B \delta w(t)
\end{align}
and note that local controlability of \eqref{eqn:ourform} is equivalent to global  controllability of \eqref{eqn:ourformloc} for any given $w$ and the corresponding trajectory $x(t)$.
Given that the linear system is controllable, one has that $[B, AB]$ is full rank.
Hence, $B$ is not an eigenvector of $A$ ($AB \neq \lambda B$ for any $\lambda \in \mathbb{R}$).
Subsequently, the function of $\lambda$ given by $|| AB - \lambda B ||$ has a minimum and is clearly both positive and unbounded above. 
To find the $\mathrm{argmin}$ $\lambda^*$, differentiate $|| AB - \lambda B ||^2$ w.r.t. $\lambda$.
It is straightforward but tedious to show that $\lambda^* = \frac{\langle AB, B \rangle}{||B||^2}$ and the corresponding minimum value of $|| AB - \lambda B ||$ is given by
\begin{align}
	m(A,B):= \frac{1}{|| B ||} \left\| AB - \frac{\langle AB, B \rangle}{{||B||}^2}\right\|.
\end{align}
This planar non-linear system is locally controllable if $A + Df$ does not have $B$ as an eigenvector at any time, as this is sufficient for the Kalman matrix to be full rank at all times (\cite{sontag2013mathematical}, Ch. 3, Corollary 3.5.18). (In fact, only a single time for which the Kalman matrix matrix is full rank is be needed, which follows from $A + Df$ not having $B$ as an eigenvector for a single time.)

Assume that $B$ is an eigenvector of $(A+Df)$, i.e., $(A+Df)B = \gamma B$ for some $\gamma \in \mathbb{R}/ \{0\}$.
Then $(AB - \gamma B) = -Df B$ and subsequently, $||AB - \gamma B|| = || Df B||$.
By the above results and the definition of a subordinate matrix norm (induced from the standard vector norm), $0 < m(A,B) \leq \|Df B\| \leq \|Df\| \|B\|$.
Rearranging gives $||Df|| \geq \frac{m(A,B)}{||B||}$ (a necessary condition for $B$ to be an eigenvector of $A+Df$).
Hence, $||Df|| < \frac{m(A,B)}{\|B\|}$ is a sufficient condition for the local surjectivity of system \eqref{eqn:ourform}.
Intuitively, this Lipschitz condition says that a locally controllable system cannot be rendered not locally controllable by the addition of a non-linear term in the state unless the non linear terms varies faster in state space (upto a constant multiple) than the linear part.

In dimensions greater than two a similar argument can be made, establishing a similar sufficient Lipschitz condition on the norm of the Jacobian of $f$.
However, an explicit form of the Lipschitz constant is not known (this has no bearing on the conclusion that local controllability holds if the condition is met).
The argument proceeds from the non-zero value of the determinant $\text{det}([B, AB, \ldots, A^{n-1}B])$ (as the system is controllable).
One can conclude from the smoothness of all maps involved that the determinant corresponding to local controllability, $\text{det}([B, (A+Df)B, \ldots, (A+Df)^{n-1}B])$, cannot be zero unless $|| Df || \geq \kappa$ for some positive constant $\kappa$.

\section{Numerical Assessment of Example Non-Linear Control Landscapes}

In this section, we give numerical simulation results for a large class of planar non-linear control systems.
We investigated the landscape critical point topology for the following class of systems
\begin{align}
	\begin{pmatrix} \dot{x}^0\\ \dot{x}^1 \end{pmatrix} = & \; A \begin{pmatrix} x^0\\ x^1 \end{pmatrix} + B w(t) + C_1 \begin{pmatrix} \cos(x^0) \\ \cos(x^1) \end{pmatrix} + S_1 \begin{pmatrix} \sin(x^0)\\ \sin(x^1) \end{pmatrix} + \nonumber \\ & \; C_2 \begin{pmatrix} \cos(2 x^0)\\ \cos(2 x^1) \end{pmatrix} + S_2 \begin{pmatrix} \sin(2 x^0)\\ \sin(2 x^1) \end{pmatrix} \label{eqn:nonlinear_example}
\end{align}
which are parameterized by square, real matrices $C_1, C_2, S_1, S_2 \in \text{Mats}(\mathbb{R}, 2)$.
Specifically, we investigated 100 randomly generated systems, all of which met the three assumptions.
This facilitated the confirmation of the conclusions of theorem \ref{sec:thm} in the cases assessed; the results obtained were consistent with these conclusions.
The simulations were based on the fidelity function $J(x) = \| x - G \|$ wherein $G$ denotes the goal.

\subsection{Assumption 1}

Systems from the class \eqref{eqn:nonlinear_example} were generated by constructing $A$ and $B$ with elements uniformly randomly chosen from the interval (-1,1).
Similarly, $C_1$, $S_1$, $C_2$, and $S_2$ were generated with elements sampled from the interval $(-0.1, 0.1)$.
The non-linear part of all of these systems is bounded and subsequently Lipschitz:
\begin{align}
& \left\| C_1 \begin{pmatrix} \cos(x^0)\\ \cos(x^1) \end{pmatrix} + S_1 \begin{pmatrix} \sin(x^0)\\ \sin(x^1) \end{pmatrix} + C_2 \begin{pmatrix} \cos(2 x^0)\\ \cos(2 x^1) \end{pmatrix} + S_2 \begin{pmatrix} \sin(2 x^0)\\ \sin(2 x^1)\end{pmatrix} \right\| \\ & \leq 2\left(||C_1|| + ||S_1|| + ||C_2|| + ||S_2|| \right). \nonumber
\end{align}
This is sufficient by the results of \cite{lukes1972global} to ensure controllability with probability one in the space of all $A, B$ irrespective of the parameter values $C_1$, $S_1$, $C_2$, and $S_2$.

\subsection{Assumption 2}

The Lipschitz constant $\frac{m(A,B)}{||B||}$ can be compared to the maximum value of $||Df||$:
\begin{align}
||Df|| & = \bigg\| C_1 \begin{pmatrix} -\sin(x^0) & 0 \\ 0 & -\sin(x^1) \end{pmatrix} + S_1 \begin{pmatrix} \cos(x^0) & 0 \\ 0 & \cos(x^1) \end{pmatrix} + \\ \nonumber & C_2 \begin{pmatrix} -2\sin(x^0)  & 0 \\ 0 & -2\sin(x^1) \end{pmatrix} + C_2 \begin{pmatrix} 2\cos(x^0)  & 0 \\ 0 & 2\cos(x^1) \end{pmatrix} \bigg\| \\ \nonumber
& \leq \sqrt{2} \left(||C_1|| + ||S_1|| + 2||C_2|| + 2||S_2|| \right)
\end{align}
to find the parameter values for which local controllability is assured to hold.
Systems generated which did not satisfy $\frac{m(A,B)}{||B||} > ||Df||$ were rejected.

\subsection{Assumption 3}

As no restriction was placed on the admissible controls during our optimizations, assumption 3 is clearly satisfied (upto numerical precision) for all systems generated.

\subsection{Optimization method}

Following the approach of \cite{joe2013topology}, successful controls were obtained using a gradient search based on the D-MORPH method \cite{rothman2005quantum}.
This method is based on traversing a smoothly parameterized set of control fields by smoothly updating the control so as to increase the fidelity.
Formally, this procedure is captured by the equation
\begin{align}
\label{eqn:dmorph}
  \frac{\partial w(s,t)}{\partial s} = \beta \frac{\delta J}{\delta w(s,t)}, \qquad \beta > 0.
\end{align}
Here, the update process is parameterized by $s$ (known as a homotopy parameter) such that $w(0,t)$ is the initial control and $\beta$ plays the role of a learning rate fixed from the outset.
In order to update the control, the right hand side of Eqn. \eqref{eqn:dmorph} needs to be evaluated.
Application of the chain rule yields 
\begin{align}
  \frac{\delta J}{\delta w(s,t)}&=
  \frac{\delta J}{\delta x(s,t)} \cdot
  \frac{\delta x(s,T)}{\delta w(s,t)}
\end{align}
where
\begin{align}
	\frac{\delta J}{\delta x(s,t)} &= 
	\frac{G - x(s,T)}{\| G - x(s,T)\|}
\end{align}
follows directly from the choosing the fidelity function to be the Euclidean distance to the goal $G$.
The second factor can be be evaluated as follows:  
\begin{align}
  \frac{\delta x(s,T)}{\delta w(s,t)} &= 
  M(s,T)M^{-1}(s,t)B.
\end{align}
Here, $M_{ij}(s,t) = \partial x_i(s,t) / \partial x_j(s,0)$ is a square, real matrix (known as the \emph{transition matrix} or \emph{propagator}) satisfying the equation
\begin{align}
	\label{eqn:diffm}
  \frac{\partial M(s,t)}{\partial t}  &= (A + Df(t))M(s,t); \qquad M(s,0) = I,
\end{align}
and is obtainable via standard integration methods.
Hence, the right hand side of Eqn. \eqref{eqn:dmorph} can be explicitly expressed as
\begin{align}
	\label{eqn:dmorphfull}
	\frac{\partial w}{\partial s} = -\beta \left( G- x(s,t) \right) \cdot \left(M(s,T)M^{-1}(s,t)B\right).
\end{align}
Note that this method uses two separate instances of numerical integration, 
both of which were implemented using Matlab's ODE45 equation solver: 
Equation \eqref{eqn:dmorphfull} is solved by integrating with respect to $s$, 
however, within each integration step 
$M(s,t)$ is obtained by integrating Eqn. 
\eqref{eqn:diffm} with respect to $t$ (for fixed $s$).

\subsection{Results}

For each system generated, 10 random goal points were generated uniformly randomly with $x,y$ coordinates in the range $(-2,2)$.
For each goal point, 10 random initial controls were generated as the cumulative sum (for the propose of smoothing) of unit amplitude white noise.
Note that the negative of the distance to the goal was used as a cost function.
For each system, goal, and initial control, the D-MORPH algorithm was run and the results recorded.
\begin{figure}[!htbp]
	\centering
    \includegraphics[width=0.85\textwidth]{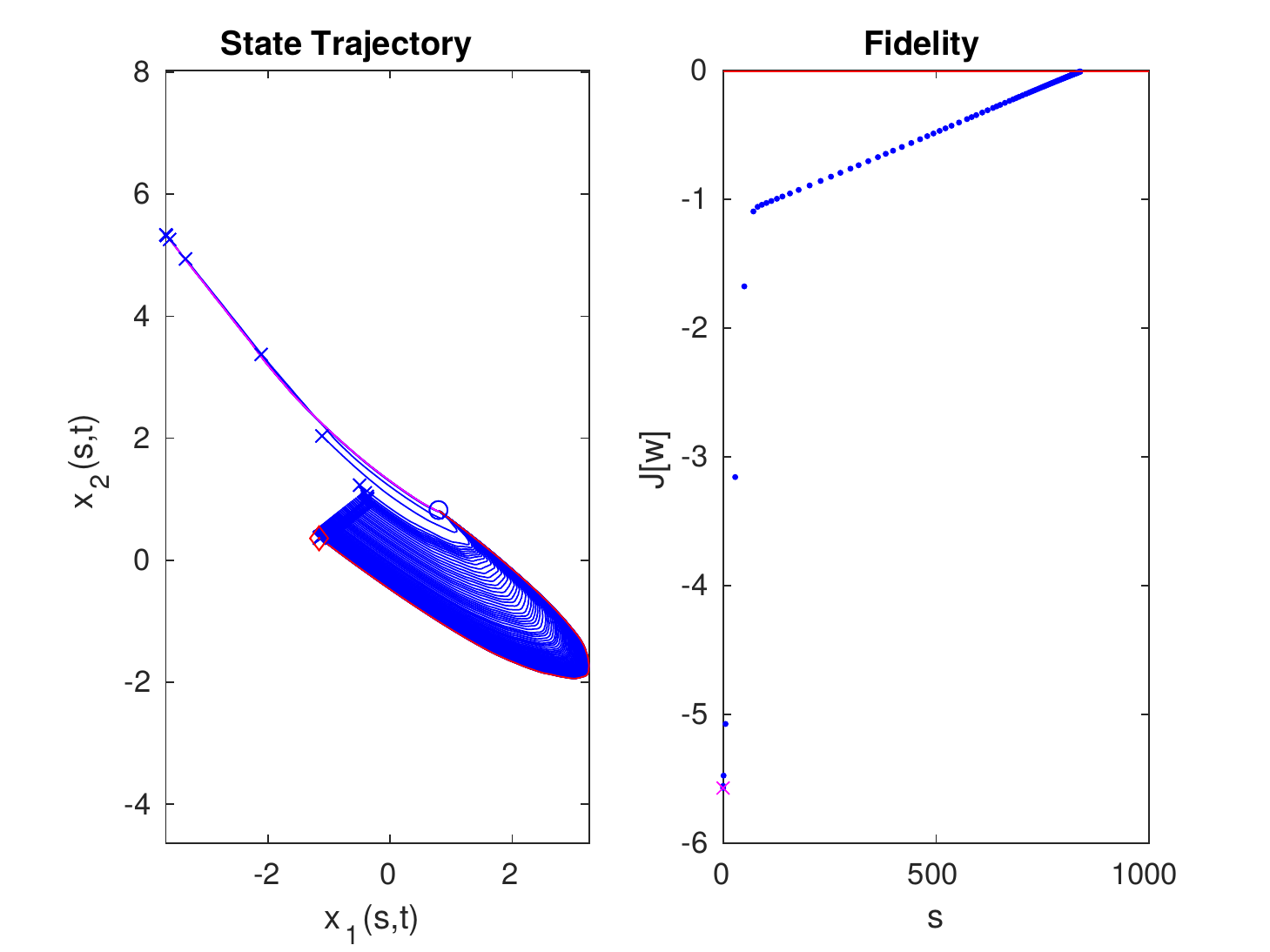}
	\caption{Trajectories (left) and the associated fidelity values (right) during a typical gradient ascent optimization.}
	\label{fig:explots1}
\end{figure}
\begin{figure}[!htbp]
  \centering
    \includegraphics[width=0.85\textwidth]{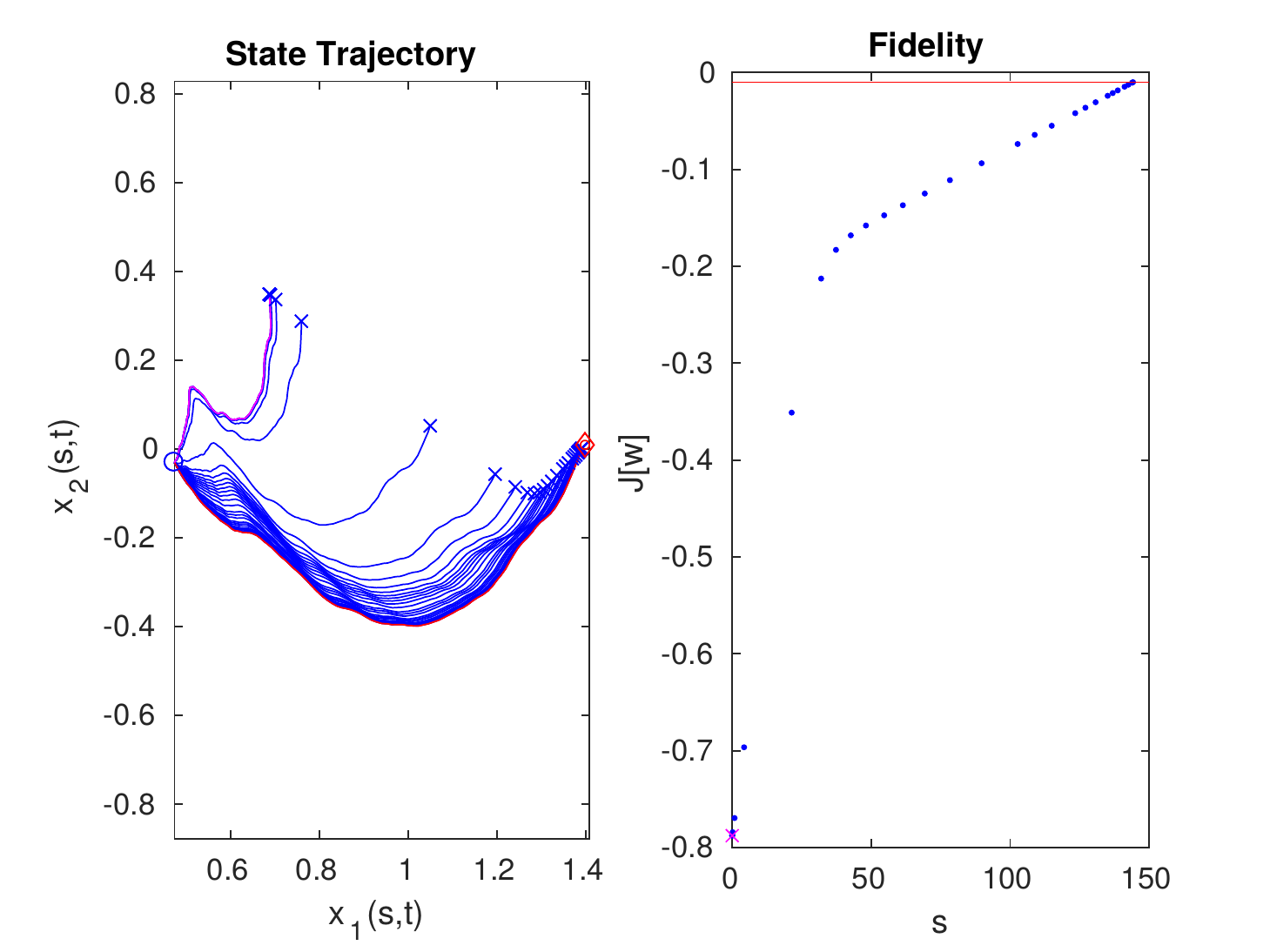}
	\caption{Trajectories (left) and the associated fidelity values (right) during a typical gradient ascent optimization.}
    \label{fig:explots2}
\end{figure}
\begin{figure}[!htbp]
  \centering
    \includegraphics[width=0.85\textwidth]{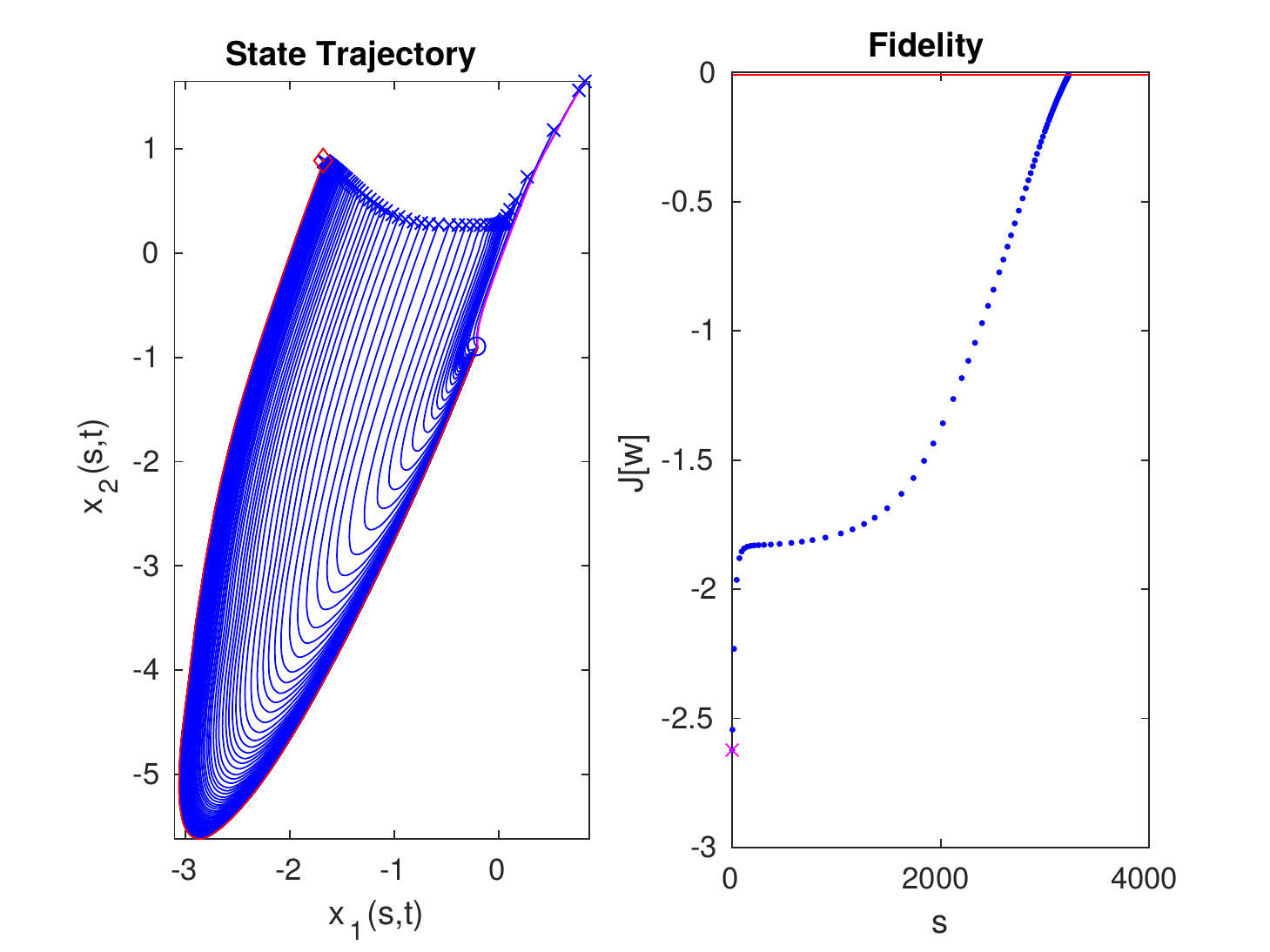}
	\caption{Trajectories (left) and the associated fidelity values (right) during a typical gradient ascent optimization.}
	\label{fig:explots3}
\end{figure}
\newpage
Within the 10,000 optimization runs completed, no landscape traps were detected.
In 96\% of optimizations the system converged to the goal directly, whereas, the remaining 4\% required further investigation.
Specifically, 2\% of optimizations timed out due to slow but progressing convergence and 2\% terminated prematurely when the fidelity decreased due to the limits of numerical precision.
In each case, a stochastic local optimization algorithm was initiated (discussed in more detail below) and the fidelity improved.
Hence, the sub-optimal controls to which gradient ascent had converged were not traps in all cases assessed.

Figures \ref{fig:explots1}, \ref{fig:explots2}, and \ref{fig:explots3} depict the trajectories and associated fidelity values during typical gradient ascent runs.
In each figure, the left panel depicts state trajectories for different values of the homotopy parameter $s$: the blue circle indicates the initial state $x_0$, the red cross indicates the goal state, the blue curves show the trajectories $x(t)$, and the blue crosses indicate the end-points.
The right panes depict the corresponding fidelity profiles and indicate monotonic convergence to the goal (the red line shows the threshold for convergence).
Figures \ref{fig:explots1}, \ref{fig:explots2} indicate rapid convergence, whereas, the flat region of \ref{fig:explots3} represents a likely saddle point (or point with small Hessian eigenvalues, i.e., a flat region) on the landscape.
These features appear to be typical for the optimization profiles of the class of systems studied numerically in this work.
\begin{figure}[!h]
  \centering
    \includegraphics[width=0.75\textwidth]{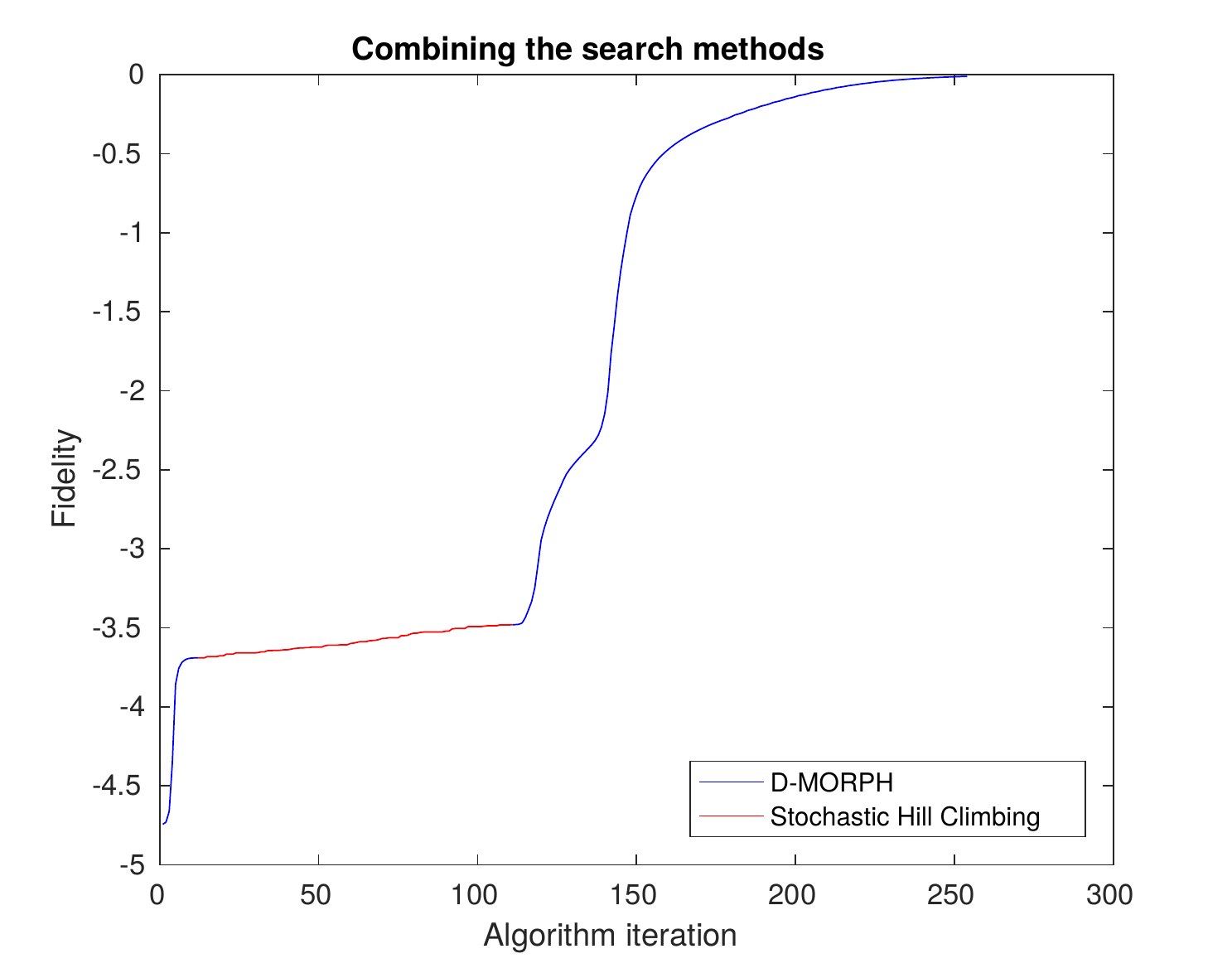}
    \caption{An optimization run plotted with algorithm iteration step on the $x$-axis and the fidelity on the y-axis. The initial blue segment represents the D-MORPH  part of the optimization, which terminated prior to reaching the goal. The red section represents the stochastic hill climbing phase, which can be seem to improve the fidelity until D-MORPH is restarted, as indicated by the second blue section which proceeds to .}
 \label{myamazinglabel}
\end{figure}
For the 4\% of optimization runs which either timed out or terminated at a sub-optimal control, a stochastic hill-climbing algorithm \cite{Wattenberg,russell2010artificial} was applied in order to assess the neighborhood of the point reached on the control landscape.
The stochastic hill-climbing algorithm used sequentially applies small, random variations to the control and evaluates the new fidelity until a variation is found which improves the fidelity; the control is then updated with this variation.
In each case the stochastic hill-climbing algorithm yielded an increase in fidelity: small but random moves in control space were able to improve the control.
This confirms that none of the runs terminated due to a trap.
Figure \ref{myamazinglabel} depicts the fidelity monotonically increasing during consecutive applications of the D-MORPH and stochastic hill-climbing algorithms for one such case.

\section{Outlook and Conclusions}

\subsection{Results Summary}
We have demonstrated that the three assumptions
\begin{enumerate}
	\item the system is fixed time controllable for a given final time $T$,
	\item the system is locally controllable,
	\item the control field(s) are unrestricted, i.e., all smooth functions $w$ on $[0,T]$ are admissible controls,
\end{enumerate}
which are known to be sufficient for trap free landscapes in quantum control, are also sufficient conditions in non-linear control systems of the form \eqref{GFOCS} when suitably expressed and when a suitable cost function $J$ is used.
We have further demonstrated that these assumptions hold for a large class of systems in the case of end-point control and with no run-time cost.
The controllability assumption generically holds (probability one) for systems with controllable linear part and which satisfy a Lipschitz condition on the non-linear part.
Similarly, the parametric transversality theorem was used to demonstrate the analogous rarity, in a specific sense, of singular critical points when considering parameterized families of control systems.
Finally, in the planar case we established a novel Lipschitz condition on the Jacobian of the non-linear part of \eqref{eqn:ourform}, sufficient for local controllability: $\|Df\| \leq \frac{\|m(A,B)\|}{\|B\|}$.
Seeking explicit formulas for such constants in the case of higher than two dimensional systems will be the topic of future work.
\begin{table}[!h]
\centering
\caption{Sufficient conditions for each assumption to hold}
\label{my-table-label}
\begin{tabular}{|l|l|}
\hline
\textbf{Assumption}   & \textbf{Sufficient Condition}                                \\ \hline
Controllable          & $f$ a Lipschitz function and $A,B$ full Kalman rank           \\ \hline
Locally Controllable  & $\|Df\|$ bounded, with constant known in the planar case       \\ \hline
Unrestricted Controls & Control space sufficient for $V_T$ to be globally surjective \\ \hline
\end{tabular}
\end{table}

\subsection{Limitations and Further Work}
In this work we restricted attention to fidelity functions depending only on the end-point and without contribution from run-time cost terms.
As discussed, run-time costs, while typically absent in quantum control, are highly prevalent in engineering applications.
Accordingly, it is important to assess the role they play in the structure of the corresponding control landscapes as this will broaden the scope of application of the results.
We also made no attempt to rigorously assess the rate of convergence during optimization, which seems to differ significantly between systems.
Work on this, especially attempting to establish criteria for efficient optimization, will also be the basis of further work.

\section*{Bibliography}
\bibliographystyle{unsrt}
\bibliography{newbib}

\end{document}